\documentclass[12pt, reqno]{amsart}
\usepackage{amssymb}
\usepackage{graphics}
\usepackage{amscd,colordvi}

\advance\hoffset by -.5truein

\newtheorem{thm}{Theorem}
\newtheorem{cor}[thm]{Corollary}
\newtheorem{prop}[thm]{Proposition}
\newtheorem{lem}[thm]{Lemma}
\theoremstyle{definition}
\newtheorem{defn}{Definition}
\newtheorem{rem}[thm]{Remark}
\newtheorem{exmp}{Example}

\def\id{\mathop{\fam 0 id}\nolimits}

\def\Span{\mathop{\fam 0 Span}\nolimits}
\def\Ker{\mathop{\fam 0 Ker}\nolimits}
\def\Cend{\mathop{\fam 0 Cend}\nolimits}
\def\QDerr{\mathop{\fam 0 Dider}\nolimits}
\def\Diend{\mathop{\fam 0 Diend}\nolimits}
\def\End{\mathop{\fam 0 End}\nolimits}

\def\Curr{\mathop{\fam 0 Cur}\nolimits}

\def\lbar{\dashv }
\def\rbar{\vdash }

\def\Hmod{H\mbox{-}\mathrm{mod}}
\def\oo#1{\mathrel{{}_{(#1)}}}

\begin{document}


\title{The Tits---Kantor---Koecher construction for Jordan dialgebras}

\author[V.~Yu. Gubarev]{V.~Yu. Gubarev$^{1}$}
\author[P.~S. Kolesnikov]{P.~S. Kolesnikov$^{2*}$}

\address{$^1$Novosibirsk State University}

\email{vsevolodgu@mail.ru}

\address{$^{2*}$Sobolev Institute of Mathematics}

\email{pavelsk@math.nsc.ru}

\begin{abstract}
We study a noncommutative generalization of Jordan algebras called
Jordan dialgebras. These are algebras that satisfy the identities
$[x_1 x_2]x_3= 0$, $(x_1^2,x_2,x_3)=2(x_1,x_2,x_1x_3)$,
$x_1(x_1^2 x_2)=x_1^2(x_1 x_2)$; they are related with Jordan algebras in the
same way as Leibniz algebras are related to Lie algebras.
We present an analogue of the Tits---Kantor---Koecher construction
for Jordan dialgebras that provides an embedding of such an algebra
into Leibniz algebra.
\end{abstract}


\maketitle

\section{Introduction}\label{sec1}

Leibniz algebras
are the most investigated non-commutative analogues of Lie algebras.
A (left) Leibniz algebra $L$ is a linear space endowed with bilinear operation
$[\cdot,\cdot ]: L\times L \to L$ that satisfies (left) Leibniz identity
\cite{L1}
\begin{equation} \label{eq:LeibIdent}
  [x,[y,z]]=[[x,y],z] + [y,[x,z]].
\end{equation}

This is well-known that an arbitrary Lie algebra $L$ can be embedded into an
appropriate associative algebra $A$ assuming the Lie bracket on $L$
coincides with the commutator
on $A$. To get a similar embedding for Leibniz algebras, J.-L.~Loday and T.~Pirashvili
in \cite{LP} proposed the notion of an (associative) dialgebra
as a substitute for the class of associative algebras.
By definition, an associative dialgebra $D$
is a linear space endowed with two bilinear operations
$\lbar , \rbar : D\times D \to D$ that satisfy certain axioms.
In particular, the ``di-commutator'' $[x,y] = x\rbar y - y\lbar x$ satisfies
\eqref{eq:LeibIdent}.

Another class of dialgebras (alternative ones) appeared in \cite{Liu}.
It is also motivated by Leibniz algebras, namely, the alternativity condition
appears as a necessary and sufficient condition for embedding a (non-associative)
dialgebra $D$ into the Steinberg Leibniz algebra $\mathrm{stl}_3(D)$, which is a non-commutative
analogue of the result from \cite{Faulk}.

In \cite{Kol1}, the natural relation between dialgebras and conformal algebras was
found. Conformal algebras were introduced in \cite{K1} as a tool for investigating
vertex algebras. Since that, the theory of conformal algebras and their generalizations
(pseudo-algebras \cite{BDK}) has been separated as an independent research area.
By definition, a pseudo-algebra $C$
is a module over a cocommutative
Hopf algebra $H$ endowed with an $H$-bilinear map
$*: C\otimes C \to (H\otimes H)\otimes_H C$
called  pseudo-product.
A general categorical approach of \cite{GK}
allows to define what is an associative
(Lie, alternative, Jordan, etc.) pseudo-algebra.

Pseudo-algebras are related to dialgebras in the following way: there exists a
functor from the category of pseudo-algebras to the category of (non-associative)
dialgebras. Under this functor, associative (alternative) pseudo-algebras turn into
associative (alternative) dialgebras, Lie pseudo-algebras turn into Leibniz algebras.
This is a reason to define what is a variety $\mathrm{Var}$ of dialgebras, where
$\mathrm{Var} $ is a homogeneous variety of ordinary algebras defined by
a family of multilinear identities.

Conversely, an arbitrary $\mathrm{Var} $ dialgebra can be embedded into an appropriate
$\mathrm{Var} $ pseudo-algebra over $H$ provided that $H$ contains a non-zero
primitive element.

This is natural to expect that
if we start with an associative dialgebra $D$
and define new operation
\[
 x\circ y =  x\rbar y + y\lbar x, \quad x,y\in D,
\]
then the algebra $(D,\circ ) $ obtained would be representative
of a class of non-associative algebras that relates to the class of
Jordan algebras in the same way as Leibniz algebras relate to Lie algebras.
This idea was implemented in \cite{VF} (where these algebras
were called quasi-Jordan algebras).
However, the adequate formalization of the class of
algebras obtained in this way should involve one more identity, as it was shown
in \cite{Br2}.
We introduce the correct notion of the variety of Jordan dialgebras
which is a subvariety of quasi-Jordan algebras from \cite{VF}.
The same notion (under the name of semi-special quasi-Jordan algebras)
was also studied in~\cite{Br09}.

There are several approaches that lead to the same notion of
Jordan dialgebras: The construction of \cite{VF}, the operadic
approach related to conformal algebras \cite{Kol1},
and the ``representational" one of \cite{Pozh}.

In Section~\ref{sec2}, we state all necessary
definitions and notations related with conformal algebras. In the
exposition, we follow \cite{BDK} and \cite{K1}, however, the axioms of
conformal algebras are adjusted for nonzero characteristic of the ground field.

In Section \ref{sec3}, it is shown how to
assign a variety of dialgebras to an arbitrary variety of ordinary algebras \cite{Pozh}.
Here we also study the relations between dialgebras and conformal algebras.
In particular, we prove that an arbitrary dialgebra can be embedded into
a current conformal algebra (which strengthen the result of \cite{Kol1}).
This embedding allows to join a bar-unit to an arbitrary dialgebra of a class
$\mathfrak M$ if the corresponding class of ordinary algebras admits joining a unit.
For associative and alternative dialgebras it was shown in \cite{Pozh}.

Solvability and nilpotency of Jordan dialgebras
are studied in Section \ref{sec4}. It is shown that a finitely generated solvable
Jordan dialgebra is nilpotent, as it happens for Jordan algebras \cite{Zh}.
Here we also state an analogue of the Pierce decomposition for
Jordan dialgebras.

The main goal of this paper is to present an analogue of the Tits---Kantor---Koecher
(TKK) construction for Jordan dialgebras that prospectively
provides an embedding of such an algebra into a Leibniz algebra.

In Section~\ref{sec5} we build the main tool
that is used to implement this construction:
the notions of a di-endomorphism and a di-derivation.
They are based on the embedding of a Jordan
dialgebra into a Jordan conformal algebra.

In Section \ref{sec6}, an analogue of the TKK construction for Jordan
dialgebras is studied. Although a similar construction
for conformal algebras (and their generalizations, pseudo-algebras)
is known \cite{Zel00, Kol09}, we can not use it directly since
it is well-defined for finite pseudo-algebras only (corresponding
to the case of finite-dimensional dialgebras).

To get rid of the condition of finite dimension, we have to
state the TKK construction independently of the conformal algebra case.
However, the embedding of a Jordan dialgebra into a Jordan
conformal algebra is still involved into consideration. But we will show
that the Leibniz algebra obtained by means of the TKK construction
does not depend on the choice
of such embedding, so we may choose the simplest one, i.e., the embedding
into current conformal algebra built in Section \ref{sec3}.
We also show that a Jordan dialgebra is nilpotent (or strongly solvable)
if and only if its TKK construction is nilpotent (or solvable) Leibniz algebra.

\section{Preliminaries on pseudo-algebras}\label{sec2}
\subsection{Pseudo-algebras over a Hopf algebra}

Let $H$ be a Hopf algebra (the main example we will use is the
polynomial algebra $H=\Bbbk [T]$ with the canonical coproduct).
Consider the class $\Hmod $ of left unital modules over the algebra $H$.
Suppose $M_1,\dots, M_n, M\in \Hmod$.
Let us say that a $\Bbbk $-linear map
\begin{equation}     \label{eq:P_nH}
a: M_1\otimes \dots \otimes M_n \to H^{\otimes n}\otimes _H M
\end{equation}
is $H^{\otimes n}$-linear if
\[
a(h_1x_1,\dots , h_nx_n) = ((h_1\otimes \dots \otimes h_n)\otimes _H 1)a(x_1,\dots, x_n)
\]
for all $h_i\in H$, $x_i\in M_i$.
Here the space
$H^{\otimes n}=\underset{n}{\underbrace{H\otimes \dots \otimes H}}$
is considered as
the  outer product of regular right $H$-modules,
i.e., this is a right $H$-module with respect to the following action:
\[
(h_1\otimes \dots \otimes h_n)\cdot T
=\sum\limits_{i=1}^n h_1\otimes \dots \otimes h_{i-1}\otimes
 h_iT \otimes h_{i+1} \otimes \dots \otimes h_n,
 \quad h_i\in H.
\]

The class $\Hmod $ is a pseudo-tensor category \cite{BD}
(or multi-category in the sense of
\cite{La})
by means of the space of multi-morphisms
$P_n^{\Hmod }(M_1,\dots , M_n; M)$ defined as the space of all
$H^{\otimes n}$-linear maps \eqref{eq:P_nH}.
The details on the composition of such maps can be found in \cite{BDK}
or \cite{Kol1}.
This multi-category is symmetric provided that $H$ is cocommutative.

Given an operad $\mathcal O$, one may define an $\mathcal O$-algebra
in $\Hmod $ in the usual way as a functor $\mathcal O \to \Hmod $.
In particular, if $\mathcal O=\mathrm {Alg}$, where
$\mathrm{Alg}$ is the operad of binary trees (that
corresponds to the variety of all non-associative algebras),
 then the pseudo-algebra structure is completely
defined by $C\in \Hmod $ equipped with a map $*\in P_2^{\Hmod}(C,C; C)$.
This map (called pseudo-product) is the image of the
elementary binary tree with
two leaves. The pair $(C, *)$ is called a pseudo-algebra over $H$ \cite{BDK}.

In particular, if $\dim_{\Bbbk } H=1$ then $\Hmod $ is just the
multi-category of linear spaces, so the notion of a pseudo-algebra
over  $H$ coincides with the ordinary notion of an algebra over~$\Bbbk $.
If $H=\Bbbk [T]$, $\mathrm{char}\,\Bbbk =0$, then pseudo-algebra over
$H$ is the same as conformal algebra \cite{K1}.

If $\mathrm{Var}$ is a homogeneous variety of algebras defined
by a family of multilinear identities, then let us denote the corresponding
quotient operad of $\mathrm{Alg}$ by $\mathrm{VarAlg}$.
As in the case of ordinary algebras (see, e.g., \cite{GK}),
a pseudo-algebra $C$ over $H$
is said to be
$\mathrm{Var}$ pseudo-algebra if the corresponding
functor $\mathrm{Alg}\to \Hmod $ can be restricted to $\mathrm{VarAlg}$.
This approach allows to define associative, alternative, Lie, Jordan, and
other classical varieties of pseudo-algebras defined by multilinear
identities. In \cite{Kol2}, it was shown that for conformal algebras
this definition agrees with the one
from \cite{Ro1} that uses coefficient algebras.

\subsection{Current conformal algebras}

Let us call by ``conformal algebras'' all pseudo-algebras over
$H=\Bbbk [T]$ without a restriction on $\mathrm{char}\,\Bbbk $.

Assume $H=\Bbbk [T]$. It is easy to see that
the outer product $H\otimes H$ is a
free right $H$-module with the basis $\{T^n\otimes 1 \}_{n\ge 0}$
\cite{BDK}.
Therefore, for every conformal algebra $C$ and for all $a,b\in C$
there exists a unique expression
\[
a*b = \sum\limits_{n\ge 0} (T^n\otimes 1)\otimes _H c_n, \quad c_n\in C.
\]
Let us denote $c_n$ by $a\oo{n} b$. Thus, the pseudo-product on $C$
is completely
defined by a countable family of operations $\oo{n}: C\otimes C \to C$
such that for all $a,b\in C$ only a finite number of $a\oo{n} b$
is nonzero (locality property).
Moreover, the condition of $H^{\otimes 2}$-linearity is equivalent
to the following properties of these operations:
\begin{equation}\label{eq:ConfAx}
\begin{gathered}
Ta\oo{n} b = a\oo{n-1}b, \ n\ge 1, \quad Ta\oo{0} b = 0, \\
T(a\oo{n} b)= a\oo{n} Tb + Ta\oo{n} b, \ n\ge 0,
\end{gathered}
\end{equation}
for all $a,b\in C$.

We will also use the following operations. Given $a,b\in C$,
denote by
$\{a*b\}$ the element
$(\sigma_{12}\otimes _H \id_C)(a*b)\in H^{\otimes 2}\otimes _H C$,
where $\sigma_{12}$ is the permutation of tensor factors in $H^{\otimes 2}$.
Analogously,
\[
\{a*b\} = \sum\limits_{n\ge 0} (T^n\otimes 1)\otimes _H \{a\oo{n} b \},
\]
where
\[
\{a\oo{n} b\} = (-1)^n \sum\limits_{s\ge 0} \binom{n+s}{s} T^s(a\oo{n+s} b).
\]
The operations $\{\cdot \oo{n} \cdot\}$
satisfy the following properties:
\begin{equation}\label{eq:ConfAxRight}
\begin{gathered}
\{a\oo{n} Tb\} = \{a\oo{n-1}b\}, \ n\ge 1, \quad \{a\oo{0} Tb\} = 0, \\
T\{a\oo{n} b\}= \{a\oo{n} Tb\} + \{Ta\oo{n} b\}, \ n\ge 0.
\end{gathered}
\end{equation}

The simplest example of a conformal algebra can be constructed as follows.
Suppose $A$ is an ordinary algebra, and
consider the free $H$-module $C=H\otimes A$ equipped
with the following pseudo-product:
\[
(f\otimes a)*(h\otimes b) = (f\otimes h)\otimes _H (1\otimes ab),
\quad f,h\in H,\ a,b\in A.
\]
In particular,
\begin{equation}\label{eq:CurrN-Prod}
(1\otimes a)\oo{n} (1\otimes b) =\begin{cases}
   1\otimes ab, & n=0, \\
   0, & n>0.
   \end{cases}
\end{equation}
Then $(C,*)$ is a conformal algebra denoted by $\Curr A$
(current conformal algebra). If $A$ belongs to a variety
$\mathrm{Var}$ then $\Curr A$ is a $\mathrm{Var} $ conformal algebra.

\subsection{Conformal endomorphisms}

In the theory of ordinary algebras, an important role
belongs to the associative algebra of linear
transformations of a linear space. The
corresponding construction for conformal algebras (and, more
generally, pseudo-algebras) was proposed in \cite{K3} and
\cite{BDK}. In this subsection, we state all necessary
definitions with a restriction to the case
$H=\Bbbk[T]$.

Consider an $H$-module $M$. A conformal endomorphism $\varphi $ of $M$
 is a linear map
\[
\varphi: M\to (H\otimes H)\otimes_H M
\]
such that
$\varphi(hx) = ((1\otimes h)\otimes _H 1)\varphi(x)$
for all $h\in H$, $x\in M$.
The space of all conformal endomorphisms of $M$ is denoted by $\Cend M$.

It is easy to see that for every $\varphi \in \Cend M$ and for every
$x\in M$ there exists a unique expression
\[
\varphi(x) = \sum\limits_{n\ge 0} (T^n\otimes 1)\otimes_H \varphi_n(x)
\]
where $\varphi_n: M\to M$, $n\ge 0$, are $\Bbbk $-linear maps
and
\begin{equation}    \label{eq:CendAx}
 \begin{gathered}
 \varphi_n(x) =0, \quad n\gg 0, \\
 \varphi_n(Tx) = T\varphi_n(x) - \varphi_{n-1}(x), \quad n\ge 0
 \end{gathered}
\end{equation}
(hereinafter, we assume $\varphi_{-1}\equiv 0$).
Hence, $\varphi \in \Cend M$ can be identified with
a sequence of $\Bbbk $-linear maps $\varphi_n: M\to M$ such that
\eqref{eq:CendAx} holds.

It is easy to check that if $\varphi\in \Cend M$, $x\in M$, then
\begin{equation}\label{eq:CendTM_Comm}
\varphi_n(T^m x) = \sum\limits_{s\ge 0}(-1)^s
     \binom{m}{s} T^{m-s}\varphi_{n-s}(x)
\end{equation}
for all $n,m\ge 0$.

The space $\Cend M$ can be considered as a left $H$-module
my means of $(T\varphi)_n = \varphi_{n-1}$ for all
$\varphi \in \Cend M$, $n\ge 0$.
This $H$-module can be equipped by operations $\oo{n}: \Cend M\otimes \Cend M
\to \Cend M$ defined as follows:
\[
 (\varphi \oo{n} \psi )_m = \sum\limits_{s=0}^n (-1)^s \binom {m+s}{s}
 \varphi_{n-s} \psi _{m+s}, \quad n,m\ge 0.
\]
These operations on $\Cend M$ satisfy \eqref{eq:ConfAx}, but, in general,
do not have the locality property. However, if $M$
is a finitely generated $H$-module then $\Cend M$ is an associative
conformal algebra \cite{BDK}.

\section{Dialgebras}\label{sec3}

\subsection{Associative and alternative dialgebras}\label{subsec3.1}

A linear space $A$ with two bilinear operations
\[
\rbar , \lbar : A\times A \to A
\]
is called a dialgebra. Various particular classes of dialgebras
introduced in literature are motivated by their relations to
Leibniz algebras. A dialgebra $A$ is associative \cite{LP}
if it satisfies the identities
\begin{equation}\label{eq:0Ident}
   (x\lbar y)\rbar z = (x\rbar y)\rbar z, \quad x\lbar (y\rbar z)= x\lbar (y\lbar z),
\end{equation}
and
\begin{equation}\label{eq:AssDiasAx}
\begin{gathered}
 (x,y,z)_\rbar \equiv (x\rbar y)\rbar z - x\rbar (y\rbar z) = 0, \\
 (x,y,z)_\lbar \equiv (x\lbar y)\lbar z - x\lbar (y\lbar z) = 0, \\
 (x,y,z)_\times \equiv (x\rbar y)\lbar z - x\rbar (y\lbar z) = 0.
\end{gathered}
\end{equation}
This class of dialgebras is well investigated in \cite{L3}.
Such dialgebras play the role of associative envelopes of Leibniz algebras:
an associative dialgebra $A$ with respect to the operation
\[
[a,b] = a\rbar b - b\lbar a, \quad a,b\in A,
\]
satisfies \eqref{eq:LeibIdent}. The Leibniz algebra
obtained is denoted by
$A^{(-)}$.

A dialgebra $A$ is said to be alternative \cite{Liu} if it satisfies
the identities \eqref{eq:0Ident} and
\begin{equation}\label{eq:AltDiasAx}
\begin{gathered}
 (x,y,z)_\rbar + (y,x,z)_\rbar = 0, \quad
 (x,y,z)_\lbar + (x,z,y)_\lbar = 0, \\
 (x,y,z)_\lbar + (y,x,z)_\times = 0, \quad
 (x,y,z)_\times + (x,z, y)_\rbar = 0.
\end{gathered}
 \end{equation}

These definitions were motivated by their relations with Leibniz
algebras. A dialgebra that satisfies the identities \eqref{eq:0Ident} is called
0-dialgebra \cite{Kol1}.
Both associative and alternative dialgebras are 0-dialgebras.

Given a 0-dialgebra $A$, the space
\[
A_0 = \Span \{ a\rbar b - a\lbar b \mid a,b\in A \}
\]
is an ideal of $A$, and $\bar A= A/A_0$ is an ordinary algebra.
The space $A$ can be endowed with the following left and right
actions of $\bar A$:
\[
\bar a\cdot x = a\rbar x, \quad x\cdot \bar a = x\lbar a, \quad x,a\in A,
\]
where $\bar a $ stands for the image of $a$ in $\bar A$.

\subsection{Varieties of dialgebras and their relation to conformal algebras}

The following definition was proposed in \cite{Pozh}.
Suppose $\mathfrak M$ is a class of ordinary algebras.
A 0-dialgebra $A$ is called an $\mathfrak M$-dialgebra if
$\bar A\in \mathfrak M$ and the split null extension
$\hat A = \bar A\oplus A$ belongs to $\mathfrak M$.

If $\mathfrak M=\mathrm{Var}$ is a homogeneous variety of algebras defined
by multilinear identities then this definition coincides with
the operadic definition of a variety of dialgebras from \cite{Kol1}.

\begin{prop}[\cite{Kol1}]
 Suppose $C$ is a $\mathrm{Var}$ conformal algebra.
 Then the same space
 $C$ equipped with new operations $a\rbar b = a\oo{0} b$,
 $a\lbar b = \{a\oo{0} b\}$, $a,b\in C$,
 is a $\mathrm{Var}$ dialgebra.
\end{prop}

This dialgebra is denoted by $C^{(0)}$.
In \cite{Kol1}, a converse statement was proved:
a $\mathrm{Var }$ dialgebra can be embedded into $C^{(0)}$
for an appropriate $\mathrm {Var}$ conformal algebra $C$.
Using the definition of \cite{Pozh}, this statement can now be strengthen
as follows.

\begin{thm}\label{thm:CurrEmbedd}
Let $\mathfrak M$ be a class of ordinary algebras. Then an arbitrary
$\mathfrak M$-dialgebra can be embedded into a current
conformal algebra over an ordinary algebra from $\mathfrak M$.
\end{thm}

\begin{proof}
Consider an $\mathfrak M$-dialgebra $A$ and let
$\hat A = \bar A\oplus A\in \mathfrak M$.
Denote $H=\Bbbk [T]$, and recall that
$\Curr \hat A = H\otimes \hat A$.
Then
the map
\[
\psi : a\mapsto 1\otimes \bar a + T\otimes a\in \Curr \hat A, \quad a\in A,
\]
is an injective homomorphism of dialgebras
$A\to (\Curr \hat A)^{(0)}$.
Indeed,
\begin{multline}\nonumber
\psi(a)\rbar \psi (b)
= (1\otimes \bar a + T\otimes a)\oo{0}  (1\otimes \bar b + T\otimes b) \\
= 1\otimes \bar a\bar b + T\otimes \bar a\cdot b
= 1\otimes \overline{a\rbar b} + T\otimes (a\rbar b) = \psi(a\rbar b)
\end{multline}
by \eqref{eq:ConfAx}, \eqref{eq:CurrN-Prod}.
The equation $\psi(a\lbar b)= \psi(a)\lbar \psi(b)$ can be proved
similarly my making use of \eqref{eq:ConfAxRight}.
\end{proof}

Recall that an element $e$ of a dialgebra $A$ is called a bar-unit \cite{Pozh},
if $e\rbar x = x\lbar e = x$ for all $x\in A$, and
$(e,x,y)_\lbar = (x,e,y)_\times = (x,y,e)_\rbar =0$ for all $x,y\in A$.
The following definition was proposed in \cite{Pozh}:
a class $\mathfrak M$-(di)algebras is called unital if
every $\mathfrak M$-(di)algebra can be embedded into
an $\mathfrak M$-(di)algebra with a (bar-)unit.

It was proved in \cite{Pozh} that the classes of
associative and alternative dialgebras are unital. Now we can
generalize this statement.

\begin{cor}\label{cor:Unital}
Let $\mathfrak M$ be a unital class of algebras. Then
the class of $\mathfrak M$-dialgebras is unital.
\end{cor}

\begin{proof}
Let $A$ be an $\mathfrak M$-dialgebra. Then $\hat A \in \mathfrak M$, so
we can find $B\in \mathfrak M$ such that $\hat A $ is a subalgebra of
$B$ and $B$ contains a unit~$e$.

Since $\Curr \hat A \subseteq \Curr B$, we have an embedding of
$A$ into $(\Curr B)^{(0)}$. This is straightforward to check
that $1\otimes e$ is a bar-unit of $(\Curr B)^{(0)}$.
\end{proof}

\subsection{Jordan dialgebras}

Let us consider the class of Jordan dialgebras over a field $\Bbbk $
such that $\mathrm{char}\,\Bbbk \ne 2,3$. In this case, the variety of
Jordan algebras is defined by multilinear identities
$x_1x_2 = x_2x_1$ and $J(x_1,x_2,x_3,x_4) = 0$,
where $J(x_1,x_2,x_3,x_4)$ is obtained by complete
linearization of the
Jordan identity $x_1(x_1^2 x_2)=x_1^2(x_1 x_2)$,
see, e.g., \cite[Section~3.3]{Zh}.

By the general scheme from \cite{Kol1}, the variety of Jordan dialgebras
is defined by \eqref{eq:0Ident} and the following identities:
\begin{equation}\label{eq:JordDialg}
\begin{gathered}
 x_1\rbar x_2 = x_2\lbar x_1, \\
 J(\dot x_1,x_2,x_3,x_4)=0,  \quad
 J(x_1, \dot x_2,x_3,x_4)=0, \\
 J(x_1,x_2,\dot x_3,x_4)=0,\quad
 J(x_1,x_2,x_3,\dot x_4)=0,
 \end{gathered}
\end{equation}
where $J(\dots, \dot x_i, \dots )$ denotes the dialgebra identity
obtained from $J$ by arranging operations
$\rbar $, $\lbar $ in such a way
that horizontal dashes are directed to the variable~ $x_i$.

The first identity in \eqref{eq:JordDialg} allows to determine a Jordan
dialgebra as an ordinary algebra with respect to the operation
$ab = a\rbar b$ (then $a\lbar b = ba$). Rewriting  \eqref{eq:JordDialg}
in terms of this operation leads to the identities
\begin{multline}\label{eq:dva}
x_1 (x_2 (x_3 x_4)) + (x_2 (x_1 x_3)) x_4
 + x_3 (x_2 (x_1 x_4)) \\
 =   (x_1 x_2)(x_3 x_4) + (x_1 x_3)(x_2 x_4)
 + (x_3 x_2)  (x_1 x_4)
\end{multline}
\begin{multline}
x_1((x_4 x_3)x_2)+x_4((x_3 x_1)x_2)+x_3((x_4 x_1)x_2) \\
=  (x_4 x_3)(x_1 x_2)+(x_1 x_3)(x_4 x_2)+(x_4 x_1)(x_3 x_2).
 \label{eq:tri}
\end{multline}
Note that the system of identities
\eqref{eq:0Ident}, \eqref{eq:dva}, \eqref{eq:tri}
is equivalent to
\begin{equation}   \label{eq:JorDiasAlg}
[x_1 x_2]x_3= 0, \quad
(x_1^2,x_2,x_3)=2(x_1,x_2,x_1x_3),
\quad
x_1(x_1^2 x_2)=x_1^2(x_1 x_2).
\end{equation}

\begin{rem}
If we apply the same scheme for Lie dialgebras, then
the variety obtained would coincide with the class of Leibniz algebras
\cite{Kol1}.
Therefore, the variety of Jordan dialgebras
relates to the variety of Jordan algebras in the same way
as Leibniz algebras relate to Lie algebras.

 In \cite{VF}, the variety of {\em quasi-Jordan algebras\/} was
 introduced as a class of algebras satisfying the first and third
 identities from \eqref{eq:JorDiasAlg}. By the reasons stated
 above, we suggest all three identities \eqref{eq:JorDiasAlg}
 to be a more adequate defining system of a  (non-commutative) dialgebra
 analogue of the Jordan algebras.
 Note that the second identity  in \eqref{eq:JorDiasAlg} was
 independently obtained in~ \cite{Br2}.
\end{rem}

\begin{exmp}
In \cite{VF}, it was shown that a quasi-Jordan algebra can be constructed
from a Leibniz algebra with an ad-nilpotent element as follows
(we state the construction for left Leibniz algebras).
If $L$ is a Leibniz algebra and $x\in L$ is an element
such that $[x,[x,[x,a]]] = 0$ for all $a\in L$, then the space
$L_x = L/\{ a\in L \mid [x,[x,a]]=0 \}$
with respect to the operation
$$
ab = [[xa]b], \quad a,b\in L,
$$
is a quasi-Jordan algebra. This is straightforward to check that the second identity
from \eqref{eq:JorDiasAlg} also holds in $L_x$, i.e.,
this is a Jordan dialgebra.
\end{exmp}

If $A$ is an arbitrary dialgebra, denote by
$A^{(+)}$ the same linear space endowed with the following
product:
\[
ab = a\rbar b + b\lbar a, \quad a,b\in A.
\]

\begin{exmp}
If $A$ is an alternative dialgebra then $A^{(+)}$ is
a Jordan dialgebra.
\end{exmp}

Indeed, by Theorem \ref{thm:CurrEmbedd}, $A$ can be embedded into
an alternative conformal algebra $C$. The anti-commutator
conformal algebra $C^{(+)} $ is a Jordan conformal algebra
\cite{Kol2}. Since the conformal 0-product in $C^{(+)}$ is defined
by $a\oo{0} b + \{b\oo{0} a \} = a\rbar b + b\lbar a$ in
$C^{(0)}$,
 $A^{(+)}$ is a subalgebra  of $(C^{(+)})^{(0)}$, i.e,
 a Jordan dialgebra.

\begin{exmp}\label{exmp:SplitNull}
Suppose $A$ is a Jordan algebra, and $M$ is a Jordan $A$-bimodule.
Then the space $J=A\oplus M$ equipped by the operation
\[
 (a+x)\otimes (b+y) \mapsto ab+ay,\quad a,b\in A, \ x,y\in M,
\]
is a Jordan dialgebra.
\end{exmp}

\begin{exmp}
Let $X$ be a finite dimensional linear space over the field
$\Bbbk$, and let $\tilde X$ be its isomorphic copy.
For $a\in X$, denote by $\tilde a $ its image in~ $\tilde X $.
Consider a symmetric bilinear form $f: X \otimes X \mapsto \Bbbk$.
Then the space
$J= \Bbbk\oplus X\oplus \tilde X$
is a Jordan dialgebra with respect to the operation
$$
(\alpha+x+\tilde a)(\beta +y+\tilde b)=\alpha
\beta+f(x+a,y+b)+\alpha y+\alpha \tilde b+\beta
(\tilde x+ \tilde a), \quad \alpha,\beta \in \Bbbk, \quad x,y,a,b \in
X .
$$
\end{exmp}

By Corollary \ref{cor:Unital}, the variety of Jordan dialgebras is unital,
so we may embed an arbitrary Jordan dialgebra $A$ into a Jordan
dialgebra $A_1$
with a left unit $e$ which belongs to the associative center of $A_1$.

\section{Solvability and nilpotency}\label{sec4}

Let $A$ be a non-associative algebra. Let us recall the following
notations:
$A^1=A^{\langle 1\rangle }=A$,
$A^n=\sum_{i=1}^{n-1}A^{n-i}A^i $,
$A^{\langle n\rangle }=A A^{\langle n -1\rangle }$.
Algebra
$A$ is said to be {\em nilpotent} (or {\em left nilpotent}) if $A^n=0$
(or $A^{\langle n\rangle }=0$) for some $n$.

Let us also define $A^{(1)}=A^2$,
$A^{(n)}=(A^{(n-1)})^2$,
$A^{[1]}=A^3$, $A^{[n]}=(A^{[n-1]})^3$.
Algebra $A$ is said to be
{\em solvable} ({\em cubic solvable}) if $J^{(n)}=0$ ($J^{[n]}=0$)
for some $n$.

Assume $J$ is a Jordan dialgebra.
Note that its solvable degrees $J^{(i)}$ are not in general
ideals of $J$, but cubic solvable degrees $J^{[i]}$ are.
However,
$J^{(2i)}\subseteq J^{[i]}\subseteq J^{(i)}$, so
solvability and cubic solvability are equivalent.

Let $\ell_a\in \End J$, $a\in J$, stands for the operator of left multiplication
in $J$, i.e., $\ell_a: x\mapsto ax$, $x\in J$. For a subset $A\subseteq J$
denote by $\ell_J (A)$ the
associative subalgebra of $\End J$ generated by all $\ell_a$, $a\in A$, and
let $\ell (J)$ stands for $\ell_J(J)$.

Defining identities of the variety of Jordan dialgebras
\eqref{eq:0Ident}, \eqref{eq:dva}, \eqref{eq:tri}
are equivalent to the following identities in $\ell (J)$:
\begin{gather}
 \label{eq:Op-raz}
  \ell_{ab}=\ell_{ba}, \\
\label{eq:Op-dva}
\ell_{t}\ell_{z} \ell_{y}+ \ell_{y}\ell_{z}\ell_{t}+\ell_{(yt)z}
  = \ell_{tz}\ell_{y}+\ell_{yt}\ell_{z}+\ell_{zy}\ell_{t}, \\
\label{eq:Op-tri}
[\ell_{y}, \ell_{tz}]+[\ell_{z}, \ell_{yt}]+[\ell_{t}, \ell_{zy}]=0.
\end{gather}

\begin{prop}\label{Thm:4_1}
A Jordan dialgebra $J$ is nilpotent if and only if it is
left nilpotent.
\end{prop}

\begin{proof}
It suffices to show that
if $J^{\langle n\rangle }=0$
then $J^{n2^{n}}=0$.
Note that $\bar J$ is a Jordan algebra, so
$\bar J^{2^n} \subseteq \bar J^{\langle n\rangle}$
\cite[Section~4.1]{Zh}.
Therefore,
$J^{m}J \subseteq J^{\langle n\rangle}J=0$
for all $m\ge 2^n$.

Now prove that $J^{k2^n}\subseteq J J^{(k-1)2^n}$
for all $k\ge 2$.
Indeed,
\[
J^{k2^{n}}=\sum_{i+j=k2^{n}} J^i J^j \subseteq
\sum_{j=1}^{(k-1)2^{n}} J^{2^{n}}J^j
+
\sum_{i=1}^{2^{n}}J^i J^{(k-1)2^{n}} \subseteq
J J^{(k-1)2^{n}}=0.
\]
In particular,
\[
 J^{n2^{n}} \subseteq J J^{(n-1)2^{n}}\subseteq
  J(J J^{(n-2)2^{n}})\subseteq\dots \subseteq
   \underset{n-1}{\underbrace{J(J(\dots  (J}} J^{2^n})\dots ))
  \subseteq J^{\langle n\rangle }=0.
\]
\end{proof}

There exist Jordan dialgebras which are right nilpotent but not left nilpotent.

\begin{exmp}
Let $X=\{x_1,x_2,\dots \}$ be a countable alphabet, $M$ be the linear
span of all words $u=x_{i_1}\dots x_{i_k}$ in $X$, $k\ge 0$
 ($u$ may be an empty word),
$1\le i_1<\dots < i_k$, and let $A$ be the linear span
of a countable set $D=\{\partial_1, \partial_2, \dots \}$.
Assume $A $ is a Jordan algebra with trivial multiplication
($\partial_i\partial_j = 0$), and consider
\[
A\otimes M \to M,
\quad
\partial_i\otimes u \mapsto
  \begin{cases}
      0, & \mbox{if $x_i$ does not appear in $u$}, \\
      (-1)^s vw, & \mbox{if $u=vx_i w$ and $s$ is the length of $v$}.
  \end{cases}
\]
Since $\partial_i(\partial_j u ) = -\partial_j(\partial_i u)$,
$M$ is a bimodule over the Jordan algebra $A$
(i.e., the split null extension $A\oplus M$ is a Jordan algebra).
Therefore, we may construct a Jordan  dialgebra $J$
as in Example~\ref{exmp:SplitNull}.
This dialgebra is right nilpotent (in particular,
solvable) since $J^2\subseteq M$, $MJ=0$.
But $J$ is not left nilpotent since
for every $n\ge 1$ we have
$\partial_{n-1}  (\dots (\partial_2(\partial_1 u))\dots )=x_n\ne 0$
if $u=x_1\dots x_n$.
\end{exmp}

\begin{thm}\label{thm:Theorem4}
A finitely generated solvable
  Jordan dialgebra is nilpotent.
\end{thm}

\begin{proof}
Let $J$ be a solvable Jordan dialgebra generated by a finite set $X$.
Since the statement is true for Jordan algebras \cite[Section~4.3]{Zh},
$J^m\subseteq [J,J]$ for some natural $m\ge 1$.
It follows from Proposition \ref{Thm:4_1} that $J$ is nilpotent if and only
if $\ell (J)$ is nilpotent.
Note that $\ell (J^m) = 0$ by \eqref{eq:Op-raz}.

The algebra $\ell (J)$ is spanned by
words of the form
\begin{equation}\label{eq:L_Op}
w= \ell_{b_1} \ell_{b_2}\ldots \ell_{b_d} \in \ell(J),\quad b_i\in J,
\end{equation}
and we may assume that $b_i$ are non-associative words in $X$
of length $k_i<m$.

Therefore, $\ell (J)$ is a homomorphic
image of an associative algebra $F\ell(J)$ generated by the set
$\{ \ell_b \mid b \in J \}$
with the defining relations \eqref{eq:Op-raz}, \eqref{eq:Op-dva},
and $\ell (J^m)=0$.
These relations
are the only conditions required
in \cite{Zh} to prove the following statement.

\begin{lem}[{\cite[Section~4.3]{Zh}}]\label{lem:ZSSS}
The algebra $F\ell(J)$ is locally nilpotent.
\end{lem}

The algebra $\ell (J) $ is the image of a subalgebra in
$F\ell (J)$
generated by $\ell _b$, $b\in Y=X\cup X^2\cup\dots \cup X^{m-1}$,
$Y$ is a finite set.
Therefore, $\ell (J)$ is nilpotent,  hence, $J$ is nilpotent.
\end{proof}

\begin{cor}\label{cor:4_3}
Let $J$ be an Jordan dialgebra and let $A$ be its
locally nilpotent subalgebra.
Then $\ell_J(A) \subseteq \ell(J) $ is locally nilpotent.
\end{cor}

Recall that a locally nilpotent radical of an algebra $A$ is
a locally nilpotent ideal $I$ such that $A/I$ has no nonzero
locally nilpotent ideals.

\begin{cor}
Every Jordan dialgebra has locally nilpotent radical.
\end{cor}

By means of Theorem \ref{thm:Theorem4},
the proof is completely similar to the one for Jordan algebras
(see, e.g., \cite[Section 4.5]{Zh}).

An element $a$ of a Jordan dialgebra $J$ is nilpotent
if there
exists $n\in \mathbb{N}$ such that $(a^n)=0$
for at least one bracketing.

Note that Jordan dialgebras are not power-associative in general.
For example, this is easy to compute that
$x(xx)\ne (xx)x$
in the the Jordan dialgebra $F^{(+)}$, where $F$ is the free associative dialgebra
generated by one variable~$x$ constructed in~ \cite{L3}.

However, the following statement holds.

\begin{cor}
If $J$ is a Jordan dialgebra and $a\in J$ is a nilpotent element
then there exists $N\in \mathbb N$ such that $(a^N)=0 $
for all bracketings.
\end{cor}

\begin{proof}
Consider the subalgebra $A$ of $J$ generated by $a$. Then
$\ell_J(A)$ is finitely generated since $\ell_J(A^n)=0$.
By Lemma \ref{lem:ZSSS} $\ell_J(A)$ is nilpotent, hence,
$A^N=0$ for some $N\ge 1$.
\end{proof}

\begin{cor}
If a finite-dimensional
Jordan dialgebra is nil then it is nilpotent.
\end{cor}

\begin{proof}
Let $J$ be a finite-dimensional Jordan nil dialgebra. Then $\bar J$ is
a Jordan algebra, so it is nilpotent \cite{Alb}. So $J^k\subseteq [J,J]$,
i.e., $J^{(k+1)}=0$. By Theorem \ref{thm:Theorem4}, $J$ is
nilpotent.
\end{proof}

\begin{cor}
A Jordan nil dialgebra $J$ of bounded index over a field of
zero characteristic is solvable. If $J$ is finitely
generated then it is nilpotent.
\end{cor}

\begin{proof}
Let $J$ be a Jordan nil dialgebra of bounded index.
Then $\bar J $ is Jordan and so it is solvable \cite{Zel}, i.e.,
$\bar J^{(k)}=0$. So for $J$ we have $J^{(k+1)}=0$.
The final statement follows from Theorem \ref{thm:Theorem4}.
\end{proof}

Let us state the Pierce decomposition for Jordan dialgebras with an idempotent.
There exists a correspondence between idempotents of a
Jordan dialgebra
$J$ and its ``algebraic image'' $\bar J$.

\begin{lem}\label{lem:4_8}
Let $\bar e $ be an idempotent in $\bar J$ for a Jordan dialgebra~ $J$,
$e\in J$, and let $h=e^2-e$. Then $f=e+2eh+h$ is an idempotent in~ $J$.
\end{lem}

\begin{proof}
Since $[J,J]J=0$, we have
\[
 f^2= (e+2eh+h)^2 = e^2 + 2e(eh) + eh = e + (eh+2e(eh))+h.
\]
It follows from \eqref{eq:dva} that
$e(eh) = e(e(e^2 - e))=e(e(ee)) - e(ee)
= \frac{1}{2} (3e(ee) - e^2) -e(ee)
= \frac{1}{2}e(e+h) - \frac{1}{2}(e+h)
= \frac{1}{2}eh$.
Therefore,  $f^2=f$.
\end{proof}

Let $e$ be an idempotent in a Jordan dialgebra $J$.
Define
$U_{a,b}=\ell_a \ell_b +\ell_b \ell_a- \ell_{ab}\in \ell(J)$
and $U_a= U_{a,a}$.
Consider the operators
$$
U_{e}=2\ell _{e}^2- \ell_e, \quad
 U_{1-e}=2\ell _{e}^2-3\ell_e+ \id_J, \quad
U_{1-e,e}=2 \ell_{e}-2 \ell_{e}^2,
$$
so that $U_e+2U_{1-e,e}+U_{1-e}=\id_J$.
Let $J_1=U_e J$,
$J_{\frac{1}{2}}=U_{1-e,e}J$,
$J_0=U_{1-e}J$.

\begin{thm}
Let $J$ be a Jordan dialgebra with an
idempotent~ $e$.
Then
 $J_i=\{x\in J \mid ex =ix\}$, $i=0,\frac{1}{2},1$,
and
$J=J_1 \oplus J_{\frac{1}{2}} \oplus J_0$.
Multiplication table for Pierce
components is the following:
\[
\begin{gathered}
J_{1}^2\subseteq J_1, \quad
J_0 J_1 + J_1 J_0= 0,
\quad J_{0}^2\subseteq J_0,\\
J_0 J_{\frac{1}{2}} +  J_{\frac{1}{2}} J_0 \subseteq J_{\frac{1}{2}},
\quad
J_1 J_{\frac{1}{2}} + J_{\frac{1}{2}} J_1 \subseteq J_{\frac{1}{2}}, \quad
J_{\frac{1}{2}}^2\subseteq J_0+J_1.
\end{gathered}
\]
\end{thm}

\begin{proof}
The equalities $J_i=\{x\in J \mid ex=ix\}$
follow from $2e(e(ex))=3e(ex) - ex$, which is a corollary of~\eqref{eq:dva}.

Relation \eqref{eq:tri} implies
$$
e^2(xy)=e(xy)=-2(ex)(ey)+x(ey)+2e((ex)y),
$$
where $x \in J_i, y \in J_k$. It gives that
$$
(2i-1)e(xy)=k(2i-1)xy.
$$
Similarly, from \eqref{eq:dva}
we obtain $(2k-1)e(xy)=i(2k-1)xy$.
As a corollary we can state
$J_1 J_k$, $J_k J_1$, $J_0 J_k$, and $J_k J_0$
are embedded into~ $J_k$. Also,
$J_0 J_1 \subseteq J_0 \cap J_1=(0)\supseteq J_1 J_0$.

It remains to prove that $J_{\frac{1}{2}}^2\subseteq J_1+J_0$,
i.e., that $U_{1-e,e}J_{\frac{1}{2}}^2 = 0$.
Indeed, \eqref{eq:dva} implies
\[
((xe)e)y+x(e(ey))+e(e(xy))=2(ex)(ey)+e(xy),
\]
so $e(e(xy))-e(xy)=0$ for all $x,y \in J_{\frac{1}{2}}$.
\end{proof}

\section{Structure Leibniz algebra}\label{sec5}

\subsection{Di-endomorphisms}
Let us fix an embedding of a Jordan dialgebra $J$
into a Jordan conformal
algebra~$C$. By $H$ we denote the polynomial algebra $\Bbbk [T]$
which has the canonical Hopf algebra structure.

The space $\Cend C$
is an associative dialgebra with respect to the following
operations:
\begin{equation} \label{eq:CendDialg}
(\varphi \rbar \psi )_n = \varphi_0\psi_n,\quad
(\varphi \lbar \psi )_n = \varphi_n\psi_0,
\end{equation}
where
$\varphi, \psi \in \Cend C$, $n\ge 0$.
Indeed, this is straightforward to check that
$\varphi \rbar \psi $
and
$\varphi \lbar \psi $ are conformal linear maps
(i.e., relations \eqref{eq:CendAx} hold), and the operations \eqref{eq:CendDialg}
satisfy \eqref{eq:0Ident}, \eqref{eq:AssDiasAx}.

Let us  also denote by $\rbar $ and $\lbar $
the following two $\Bbbk $-linear maps:
\[
\begin{gathered}[]
\rbar ,\lbar : \Cend C\otimes C \to C, \\
\varphi \rbar a = \varphi_0(a), \quad
\varphi \lbar a = \sum\limits_{n\ge 0} T^n\varphi_n(a), \\
\varphi\in \Cend C, \quad a\in C.
\end{gathered}
\]

\begin{lem}\label{lem:CendRel1}
For all $\varphi, \psi \in \Cend C$, $a\in C$ we have
\[
\begin{gathered}[]
 (\varphi\rbar \psi )\rbar a= (\varphi\lbar \psi )\rbar a= \varphi \rbar (\psi \rbar a), \\
 (\varphi\rbar \psi )\lbar a= \varphi \rbar (\psi \lbar a), \\
 (\varphi\lbar \psi )\lbar a = \varphi \lbar (\psi \lbar a) =
 \varphi \lbar (\psi \rbar a).
\end{gathered}
\]
\end{lem}

\begin{proof}
It follows from the definitions, that
\[
(\varphi\rbar \psi )\rbar a= (\varphi\lbar \psi )\rbar a= \varphi \rbar (\psi \rbar a)
=\varphi_0\psi_0(a),
\]
\[
(\varphi\rbar \psi )\lbar a= \varphi \rbar (\psi \lbar a)=
\sum\limits_{n\ge 0} T^n(\varphi_0\psi_n(a)).
\]
Let us check the last relation:
\[
(\varphi\lbar \psi )\lbar a=
\sum\limits_{n\ge 0} T^n((\varphi\lbar \psi )_n(a))
=
\sum\limits_{n\ge 0} T^n(\varphi_n\psi _0(a))=\varphi \lbar (\psi \rbar a).
\]
On the other hand,
\[
\varphi \lbar (\psi \lbar a)
=\sum\limits_{n\ge 0}T^n\varphi_n(\psi\lbar a)
= \sum\limits_{n\ge 0}T^n\varphi_n\left(\sum\limits_{m\ge 0} T^m\psi_m(a) \right ).
\]
It follows from \eqref{eq:CendTM_Comm} that
\[
\varphi \lbar (\psi \lbar a) =
\sum\limits_{n,m,s\ge 0}(-1)^s \binom{m}{s} T^{n+m-s}
 \varphi_{n-s}\psi_m(a).
\]
Assuming $t=n-s\ge 0$, we obtain
\begin{multline}\nonumber
 \varphi \lbar (\psi \lbar a) =
\sum\limits_{t,m,s\ge 0}(-1)^s T^{t+m}\binom{m}{s}
 \varphi_{t}\psi_m(a) \\
 = \sum\limits_{t,m\ge 0}
 \left(\sum\limits_{s\ge 0}(-1)^s\binom{m}{s}\right)
T^{t+m} \varphi_{t}\psi_m(a)
= \sum\limits_{t \ge 0} T^{t} \varphi_{t}\psi_0(a).
\end{multline}
This proves the last equality.
\end{proof}

\begin{rem}
In the case when $C$ is finitely generated over $H$ (e.g., if $\dim J<\infty$) then
the last lemma can also be derived from the associativity of the pseudo-algebra
$\Cend C$. If $C$ is not a finite $H$-module then $\Cend C$
is not a pseudo-algebra,
so we have to prove the statement explicitly.
\end{rem}

\begin{lem}\label{lem:CendRel2}
Suppose $C$ is a pseudo-algebra, $\varphi \in \Cend C$, $x,y\in C$.
Then the following relations hold in the dialgebra $C^{(0)}$:
\[
(\varphi \lbar x)\rbar y  = (\varphi \rbar x)\rbar y ,\quad
x\lbar (\varphi \rbar y)  = x\lbar (\varphi \lbar y).
\]
\end{lem}

\begin{proof}
Let us check the second relation. By definition,
\[
\varphi\lbar y =\sum\limits_{n\ge 0} T^n\varphi_n(y),
\]
but since $x\lbar Tz = 0$, we have
$x\lbar (\varphi \lbar y)= x\lbar\varphi_0(y)
 = x\lbar (\varphi \rbar y)$.
\end{proof}

Consider operators $L_a\in \Cend C$ of left pseudo-multiplication on
$a\in J\subseteq C$. Namely,
\[
 L_a: x\mapsto a*x\in (H\otimes H)\otimes _H C, \quad x\in C.
\]
In particular, if $x\in J$ then $(L_a)_0: x\mapsto L_a\rbar x = ax\in J$.
Since $C$ is a commutative pseudo-algebra,
\[
L_a(x) = a*x = (\sigma_{12}\otimes_H\id_C)(x*a),
\]
so
\[
L_a\lbar x = L_x\rbar a = xa\in J, \quad x\in J.
\]

\begin{defn}\label{defn:Diend}
Given an embedding of $J$ into a pseudo-algebra $C$,
define the space of {\em di-endomorphisms\/} of $J$ as
\[
\Diend_C J = \{\varphi \in\Cend C \mid
\varphi \rbar J, \varphi \lbar J\subseteq J \}/
  \{\varphi \in \Cend C\mid \varphi \rbar J = \varphi \lbar J =0 \}.
\]
\end{defn}

We will denote $\Diend_C J$ by  $\Diend J$ when the embedding of $J$
into $C$ is fixed.

For example, the images of
operators of left pseudo-multiplication
$L_a\in \Cend C$, $a\in J$, are di-endomorphisms of $J$; we will
 also denote them by~ $L_a$.
Let $L(J)$ be the linear subspace in $\Diend J$
spanned by all operators $L_a$, $a\in J$. If $A$ is a subspace of $J$,
then $L_J(A)=L(A)$ stands for the subspace $\{L_a\mid a\in A\}\subseteq L(J)$.

It is clear that the operations
$\rbar $, $\lbar $ are correctly defined on
$\Diend J\otimes \Diend J \to \Diend J$
and on
$\Diend J\otimes J \to J$.
Then all the properties proved in Lemmas \ref{lem:CendRel1},~\ref{lem:CendRel2}
above hold for $\Diend J$ instead of $\Cend C$.
In particular, $\Diend J $ is an associative dialgebra, and
$(\Diend J)^{(-)}$ is a Leibniz algebra.

For the Leibniz bracket $[L_a,L_b]= L_a\rbar L_b - L_b\lbar L_a$
in $(\Diend J)^{(-)}$, $a,b\in J$, we have
\begin{equation}\label{eq:LBrack_on_L}
[L_a,L_b]\rbar x = a(bx)-b(ax),
\quad
[L_a,L_b]\lbar x = a(xb)-(xa)b = -(a,x,b),
\quad
x\in J,
\end{equation}
by Lemma~\ref{lem:CendRel1}.

\begin{lem}\label{lem:CommRelations}
The following relation holds in $(\Diend J)^{(-)}$:
\[
[L_{ab}, L_c] = [L_b, L_{ac}] + [L_a, L_{bc}],\quad a,b,c \in J.
\]
\end{lem}

\begin{proof}
Let $D = [L_{ab}, L_c] - [L_b, L_{ac}] - [L_a, L_{bc}] \in \Diend J$.
It is sufficient to check that
$D\rbar d = D\lbar d = 0 $ for all $d\in J$.
Indeed, it follows immediately from \eqref{eq:LBrack_on_L}
that
\[
D\rbar d
= (ab)(cd) - c((ab)d) - b((ac)d) + (ac)(bd) - a((bc)d) + (bc)(ad),
\]
which is zero due to \eqref{eq:tri}, and
\[
D\lbar d
= (ab)(dc) - (d(ab))c - b(d(ac)) + (db)(ac) - a(d(bc)) + (ad)(bc),
\]
which is zero due to \eqref{eq:dva}.
\end{proof}

\subsection{Di-derivations}

\begin{defn}\label{defn:QuasiDer}
A di-endomorphism $D \in \Diend J$ is called a
 {\em di-derivation\/} of $J$ if
\begin{equation}\label{eq:QDerDefn}
D\rbar (xy) = (D\rbar x) y + x (D\rbar y), \quad
D\lbar (xy) = y(D\lbar x) + x (D\lbar y)
\end{equation}
for all $x,y\in J$.
\end{defn}

Denote by $\QDerr(J)=\QDerr_C(J)\subseteq \Diend J$
the space of all di-derivations of $J$
with respect to the embedding $J\subseteq C^{(0)}$.

\begin{prop}\label{prop:QDerComm}
Let $a,b\in J$,
Then $D=[L_a,L_b]\in \Diend J$
is a di-derivation of~$J$.
\end{prop}

\begin{proof}
It follows from \eqref{eq:LBrack_on_L}
that we have to deduce the following
identities to hold in Jordan dialgebras:
\[
  (b(a x))y + a(b(x y)) + x(b(a y))   =
  (a(b x))y + b(a(x y)) + x(a(b y)) ,
\]
\[
(a,xy,b) = x(a,y,b) + y(a,x,b).
\]
The first one means that the commutator of two
operators of left multiplication in Jordan dialgebras is a usual
derivation. This is easy to deduce from \eqref{eq:dva}.

The second identity can be obtained as follows.
Partial linearization of the third identity of \eqref{eq:JorDiasAlg}
leads to
\[
 2(xa)(xb) + x^2(ab) = 2x((xa)b) + a(x^2b).
\]
Subtract $(x^2a)b=(ax^2)b$ from the left and right parts to get
\[
 2(xa)(xb) - (x^2,a,b) = 2x((xa)b) - (a,x^2,b).
\]
Now use \eqref{eq:JorDiasAlg} to obtain
\[
(a,x^2,b)=2x(a,x,b).
\]
Linearization of this identity leads to the required relation
by means of the first identity of \eqref{eq:JorDiasAlg}.
\end{proof}

\begin{rem}
In the case when $C$ is a finite $H$-module,
the last proposition can also be derived from the fact that
$D=[L_a*L_b]$ is a pseudo-derivation of the Jordan pseudo-algebra~$C$
\cite{Kol09},
i.e., $D$ satisfies the condition
\[
 D(a*b) = D(a)*b + (\sigma_{12}\otimes_H \id_C)(a*D(b))\in H^{\otimes 3}\otimes_H C.
\]
\end{rem}

\begin{lem}\label{lem:QDerLeib}
The space $\QDerr(J)$ is a Leibniz subalgebra of $(\Diend J)^{(-)}$.
\end{lem}

\begin{proof}
It suffices to check that if $D_1, D_2\in
\QDerr (J) $ then $D=[D_1,D_2]= D_1 \rbar D_2 - D_2 \lbar D_1$
also belongs to $\QDerr (J)$.

It follows from Lemma \ref{lem:CendRel1}  that
$(D_1 \vdash D_2 - D_2 \dashv D_1)\vdash (xy)=(D_1 \vdash D_2 - D_2
\vdash D_1)\vdash (xy)$,
and it is well known that
commutator of two derivations of ordinary algebra is
again a derivation.
Also, we have
\begin{multline}\nonumber
D \dashv (xy)
= D_1\vdash(D_2\dashv (xy)) - D_2 \dashv (D_1\dashv (xy))  \\
= D_1\vdash x(D_2\dashv y) + D_1\vdash y(D_2\dashv x)
  - D_2 \dashv x(D_1\dashv y)- D_2 \dashv y(D_1\dashv x) \\
= x (D_1\vdash (D_2\dashv y)) + (D_1\vdash x)(D_2\dashv y)
  +
  y (D_1\vdash (D_2\dashv x)) + (D_1\vdash y)(D_2\dashv x)\\
  -
  x(D_2 \dashv (D_1\dashv y)) -(D_1\dashv y)(D_2\dashv x)
  -
  y(D_2 \dashv (D_1\dashv x)) -(D_1\dashv x)(D_2\dashv y) \\
= x (D\dashv y) + y (D\dashv x)
\end{multline}
since
$(D_i\lbar a)b = (D_i\rbar a)b$ by Lemma \ref{lem:CendRel2}.
\end{proof}

\subsection{Structure algebra}

Let
\[
\mathrm S(J)= L(J)\oplus \QDerr(J)
\]
be the formal direct sum of two subspaces of $\Diend J$.
Define the following operation
$[\cdot, \cdot ]_s$ on
$\mathrm S(J)$:
\begin{equation}\label{eq:StructOperation}
\begin{gathered}[]
[L_a,L_b]_s = [L_a,L_b]\in \QDerr (J),\quad
[D,L_a]_s= L_{D\rbar a}\in L(J), \\
[L_a, D]_s = -L_{D\lbar a}\in L(J), \quad [D_1,D_2]_s = [D_1,D_2]\in \QDerr(J)
\end{gathered}
\end{equation}
for $a,b\in J$, $D_1,D_2\in \QDerr (J)$.

\begin{thm}\label{thm:StructureLeibnizAlg}
The space $\mathrm S(J)$ is a Leibniz algebra with respect to the operation
\eqref{eq:StructOperation}.
\end{thm}

\begin{proof}
It is enough to make sure that the Leibniz identity
$[x,[y,z]_s]_s = [y,[x,z]_s]_s + [[x,y]_s,z]_s$ holds
for all $x,y,z \in \mathrm S(J)$.
This can be done in a straightforward way by making use
of the following equalities:
\begin{equation}\label{eq:StructureRelations}
\begin{gathered}[]
[L_a,L_c]\lbar b - [L_b,L_c]\lbar a = [L_a,L_b]\rbar c, \\
L_{D\rbar a} = [D, L_a], \quad L_{D\lbar a} = -[L_a, D],
\end{gathered}
\end{equation}
where $a,b,c\in J$, $D\in \QDerr (J)$, and
the commutators are computed in $(\Diend J)^{(-)}$.

Indeed,
$[L_a,L_c]\lbar b - [L_b,L_c]\lbar a=
-(ab)c + a(bc) + (ba)c - b(ac)=
a(bc) - b(ac) = [L_a,L_b]\rbar c$
by \eqref{eq:LBrack_on_L}.

Further, $[D, L_a]\rbar x =
D\rbar (ax) - (L_a\lbar D)\rbar x=
(D\rbar a)x + a(D\rbar x) - L_a\rbar (D\rbar x) = (D\rbar a ) x
= L_{D\rbar a}\rbar x$, $x\in J$.
Also, for all $x\in J$ we have
$[D, L_a]\lbar x = D\rbar (L_a\lbar x) - L_a\lbar (D\lbar x)
= D\rbar (xa) - (D\lbar x)a =  (D\rbar x)a + x(D\rbar a)  - (D\lbar x)a
=L_{D\rbar a}\lbar x$
by Lemma \ref{lem:CendRel2}.
The last equality can be proved
in a similar way.
\end{proof}

Let $\mathrm S_0(J)$ stands for the subspace of $\mathrm S(J)$ spanned
by $L_{ab}\in L(J^2)$ and $[L_a,L_b]\in \QDerr(J)$ for all $a,b\in J$.
Note that this subspace is closed under the Leibniz bracket
$[\cdot, \cdot ]_s$.
Therefore, $\mathrm S_0(J)$ is a Leibniz subalgebra
of $\mathrm S(J)$, which is called the {\em structure algebra\/} of~$J$.

Note that the construction of $\mathrm S_0(J)$ is based on the
embedding $J\subseteq C^{(0)}$ which has been fixed in the
very beginning. Let us show that the structure algebra
$\mathrm S_0(J)$
does not actually depend on the choice of~$C$.

\begin{lem}\label{lem:Independence1}
Let $C_1$ and $C_2$ be Jordan conformal algebras such that
$J\subseteq C_i^{(0)}$, $i=1,2$. Assume there exists
an epimorphism of conformal algebras $\tau : C_1\to C_2$
such that $\tau $ acts on $J$ as the identity map.
Denote by $\mathrm S_0^i(J)$  the structure Leibniz
algebra based on the embedding of $J$ into $C_i^{(0)}$, $i=1,2$.
Then $\mathrm S_0^{1}(J) \simeq \mathrm S_0^{2}(J)$.
\end{lem}

\begin{proof}
Recall that $\mathrm S_0^i (J)$ is a Leibniz subalgebra
of $\mathrm S^i(J) = L^i(J)\oplus \QDerr^i (J)$,
where $L^i(J)$ is spanned by the di-endomorphisms
$L^i_a\in \Diend_{C_i} J$, $a\in J$,
$\QDerr^i(J) = \QDerr_{C_i} (J)$.

The desired isomorphism is supposed to be defined as
$L^1_{ab} \mapsto L^2_{ab}$, $[L^1_a, L^1_b] \mapsto [L^2_a, L^2_b]$,
$a,b\in J$. To prove that this map is in fact a well-defined isomorphism
of Leibniz algebras, it is enough to show that the
associative dialgebras $A_i$, $i=1,2$, generated in $ \Diend_{C_i} J $
by the operators $L^i_a$, $a\in J$, are isomorphic.

Denote by $B_i$ the associative dialgebra generated by
all $L^i_a \in \Cend C_i$, $a\in J$, $i=1,2$. Then
$A_i = B_i/\{b\in B_i\mid b\rbar J=b\lbar J=0 \}$.
Let $\pi_i$ stands for the natural map $B_i\to A_i$.

Let
$F= Fd\langle J \rangle $ be the free associative dialgebra generated by
$J$ as by a set of generators,
$\tau_i : F \to B_i\subseteq (\Cend C_i)^{(0)}$,
$i=1,2$,
be the epimorphisms of dialgebras
given by $a \mapsto L^i_a \in \Cend C_i$, $a\in J$.

It was shown in \cite{L3} that every element of $F$ can be
uniquely presented as a linear combination of monomials of the form
$\dot u_j =x_1\dots \dot x_j \dots x_n$,
$n\ge 1$, $1\le j\le n$, $u=x_1\dots x_n$ is a word in $J$,
where the dot over $x_j$ means that the dialgebra operations
in $\dot u_j$ are arranged in such a way that horizontal dashes are
directed at $x_j$.
Note that if $u=x_1\dots x_n$ then
\[
\tau_i(\dot u_j) = L^i_{x_1}\rbar \dots \rbar  L^i_{x_j}
 \lbar \dots \lbar L^i_{x_n}.
\]
Therefore,
\[
 \tau(\tau_1(f)_{m} x) =  \tau_2(f)_{m} \tau(x)
\]
for all $f\in F$, $m\ge 0$, $x\in C_1$. Since
$\tau $ is surjective, $\Ker \tau_1 \subseteq \Ker
\tau_2$, and there exists an epimorphism
$\psi : B_1\to B_2$
such that $\psi(L^1_a) = L^2_a$, $a\in J$.
In particular,
\[
\psi(b)_{m} \tau(x) = \tau(b_{m} x), \quad b\in B_1, \ m\ge 0,\ x\in C_1.
\]

Consider the composition $\psi\circ \pi_2: B_1\to A_2$.
Note that $b\rbar J, b\lbar J\subseteq J$ for every $b\in B_1$.
Therefore,
\[
\psi(b)\rbar x  = b\rbar x \in J, \quad
\psi(b)\lbar x  = b\lbar x \in J
\]
for every $x\in J$, and
$\Ker (\pi_2\circ \psi)=\Ker \pi_1$.
Hence, $A_1\simeq A_2$.
\end{proof}

\begin{prop}\label{prop:Independence}
The structure algebra $\mathrm S_0(J)$ does not
depend on the choice of embedding $J\subseteq C^{(0)}$.
\end{prop}

\begin{proof}
Suppose $C_1$ and $C_2$ are two Jordan conformal algebras
such that $J\subseteq C_i^{(0)}$, $i=1,2$.
Consider the Cartesian product $C=C_1\times C_2$, which is
also a Jordan conformal algebra, and $J\subseteq C^{(0)}$
in the obvious way: $a\mapsto (a,a)\in C$.

The canonical projections $\pi_i: C\to C_i$ are epimorphisms
of conformal algebras such that $\pi_i\vert _J =\id$.
By Lemma \ref{lem:Independence1}, $\mathrm S_0^i(J)\simeq \mathrm S_0(J)$,
where $\mathrm S_0(J)$ stands for the structure Leibniz algebra based
on the embedding $J\subseteq C^{(0)}$. Therefore,
$\mathrm S^1_0(J) \simeq \mathrm S_0^2(J)$.
\end{proof}

\section{The Tits---Kantor---Koecher construction}\label{sec6}
\subsection{Super-structure Leibniz algebra}

Suppose $J$ is a Jordan dialgebra as above, and let
$\mathrm S_0(J)$ be its structure Leibniz algebra built in the
previous section.

By the definition, $\mathrm S_0(J)\subseteq \Diend_C J\oplus \Diend_C J$
for some Jordan conformal algebra $C$ such that $J\subseteq C^{(0)}$.
Neither the structure of $\mathrm S_0(J)$, nor the operations
\[
\begin{gathered}
 \rbar , \lbar : \mathrm S_0(J)\otimes J \to J, \\
 (L_x + [L_y,L_z])\rbar a = xa + [L_y,L_z]\rbar a =
 xa + y(za) - z(ya), \\
 (L_x + [L_y,L_z])\lbar a = ax + [L_y,L_z]\lbar a =
 ax + y(az) - (ay)z,
\end{gathered}
\]
depend on the choice of $C$. It follows from relations
\eqref{eq:StructOperation},
\eqref{eq:StructureRelations},
and Lemma~\ref{lem:CendRel1} that
\begin{equation} \label{eq:StructureCommutator}
 \begin{gathered}[]
 [U,V]_s \rbar a  = U\rbar (V\rbar a) - V\rbar (U\rbar a) , \\
 [U,V]_s \lbar a  = U\rbar (V\lbar a) - V\lbar (U\lbar a)  \\
 \end{gathered}
\end{equation}
for all $U, V\in \mathrm S_0(J)$, $a\in J$.

Consider the formal direct sum
\[
\mathrm T(J) = J^+ \oplus \mathrm S_0(J) \oplus J^-,
\]
where $J^{\pm }$ are isomorphic copies of the space $J$.
The image of an element $a\in J$ in $J^\pm$ we will denote by $a^\pm $.

Define a bilinear operation $[\cdot , \cdot ]_t$ on
$\mathrm T(J)$ as follows:
\begin{equation}\label{eq:TKKoperations}
\begin{gathered}[]
[a^+, b^+]_t = [a^-,b^-]_t =0, \quad [D,U]_t = [D,U]_s, \\
[a^+, b^-]_t = -L_{ab} + [L_a,L_b], \quad
[a^-, b^+]_t = L_{ab} + [L_a,L_b],  \\
[a^-, D]_t = - (D\lbar a)^-, \quad
[a^+, D]_t = - (D^*\lbar a)^+, \\
[D, a^-]_t = (D\rbar a)^-, \quad
[D, a^+]_t = (D^*\rbar a)^+,
\end{gathered}
\end{equation}
where $a,b \in J$, $D, U\in \mathrm S_0(J)$,
and if $D=L_x + [L_y,L_z]$ then $D^*$ stands for $-L_x+ [L_y,L_z]$.

It is useful to note that the map
\[
a^+ + D + b^- \mapsto b^+ + D^* + a^-, \quad a,b\in J,\ D\in \mathrm S_0(J),
\]
is an automorphism of $\mathrm T(J)$.

\begin{thm}
The space $\mathrm T(J)$ equipped with the operation
\eqref{eq:TKKoperations} is a $\mathbb Z_3$-graded Leibniz
algebra.
\end{thm}

\begin{proof}
It is easy to see from \eqref{eq:TKKoperations} that $\mathrm T(J)$
is indeed a $\mathbb Z_3$-graded algebra with homogeneous components
$J^+$, $\mathrm S_0(J)$, $J^-$.

This is straightforward to compute
that the Leibniz identity \eqref{eq:LeibIdent} holds for the operation $[\cdot, \cdot]_t$.
Let us consider some examples in order to derive necessary relations.

(1) For $x=a^+$, $y= b^+$, $z=c^-$, $a,b,c\in J$, we have
\[
[x,[y,z]_t]_t = [a^+, -L_{bc} + [L_b, L_c]]_t
= (-L_{bc}\lbar a - [L_b, L_c]\lbar a)^+
=( - a(bc) - b(ac) + (ab)c)^+  ,
\]
so
\[
[y,[x,z]_t]_t = (-b(ac) - a(bc) + (ba)c)^+, \quad [[x,y]_t, z]_t =0,
\]
and \eqref{eq:LeibIdent} holds.

(2) For $x=a^+$, $y = b^-$, $a,b\in J$, and $z= U\in \mathrm S_0(J)$,
we have
\[
[x,[y, z]_t]_t = - [a^+, (U\lbar b)^-]_t
  = L _{a(U\lbar b)} - [L_a, L_{U\lbar b}],
\]
\[
[y,[x,z]_t]_t = - L_{b(U^*\lbar a)} - [L_b, L_{U^*\lbar a}],
\]
\[
[[x,y]_t, z]_t = -[L_{ab}, U]_s + [[L_a, L_b],U]_s.
\]
Consider two cases:  $U\in L(J)$, and $U\in \QDerr (J)$.
In the first case, $U=L_c$, $c\in J$, $U^*=-U$, so
\[
 [x,[y, z]_t]_t = L_{a(bc)} - [L_a, L_{bc}],
 \quad
 [y,[x,z]_t]_t = L_{b(ac)} + [L_b, L_{ac}],
\]
\[
 [[x,y]_t, z]_t =  L_{[L_a,L_b]\rbar c} - [L_{ab}, L_c].
\]
Therefore, the required relation follows from
\eqref{eq:LBrack_on_L} and
Lemma \ref{lem:CommRelations}.

In the second case, $U^*=U$, so
\[
[[x,y]_t, z]_t = L_{U\lbar (ab)} + [[L_a, L_b],U].
\]
Then \eqref{eq:LeibIdent} follows from \eqref{eq:StructOperation},
Definition \ref{defn:QuasiDer}, and Theorem~ \ref{thm:StructureLeibnizAlg}.

(3) For $x=U$, $z = V$, $U,V \in \mathrm S_0(J)$,
$y=a^+$, $a\in J$, we have
\[
[x,[y, z]_t]_t = - [U, (V^*\lbar a)^+]_t
= -(U^*\rbar (V^*\lbar a))^+,
\]
\[
[y,[x,z]_t]_t = [a^+, [U,V]_s]_t = -([U,V]_s^*\lbar a)^+,
\]
\[
[[x,y]_t, z]_t = [(U^*\rbar a)^+, V]_t = -(V^*\lbar (U^*\rbar a))^+.
\]
Then \eqref{eq:StructureCommutator} implies the required relation
by means that $U\mapsto U^*$ is an automorphism of $\mathrm S_0(J)$.

For other choice of $x$, $y$, $z$ from homogeneous components of $\mathrm T(J)$,
the Leibniz identity can be verified in a similar way.
\end{proof}

The Leibniz algebra $\mathcal L= \mathrm T(J)$ is called {\em super-structure\/}
algebra of $J$. This is a non-commutative analogue
of the Tits---Kantor---Koecher construction.
Since both $J$ and $\mathcal L$ are 0-dialgebras,
it is reasonable to explore relations between the Jordan
algebra $\bar J$ and the Lie algebra $\bar{\mathcal L}$.

\begin{thm}\label{thm_barfunctor}
If a Jordan dialgebra $J$ contains a bar-unit then
$\overline{\mathrm T(J)}\simeq \mathrm T(\bar J)$.
\end{thm}

\begin{proof}
Suppose $J_1$ and $J_2$ are two
Jordan dialgebras, $\varphi $ is a homomorphism from $J_1$ onto $J_2$.
Then there exists a homogeneous epimorphism of Leibniz algebras
$\tau =\mathrm T(\varphi ): \mathrm T(J_1) \to \mathrm T(J_2)$ given by
\[
\tau (a^{\pm}) = \varphi(a)^{\pm}, \quad
\tau(L_a)  = L_{\varphi(a)},\quad
\tau ([L_a,L_b])= [L_{\varphi(a)}, L_{\varphi(b)}],
\quad a,b\in J_1.
\]
In particular, for a Jordan dialgebra $J$, the natural epimorphism
$J\to \bar J$ gives rise to an epimorphism
$\tau: \mathrm T(J) \to \mathrm T(\bar J)$. Since $\bar J$ is an ordinary
Jordan algebra, $\mathrm T(\bar J)$ is a Lie algebra (the ordinary TKK construction).
Thus,
\[
\Ker \tau \supseteq I_0 : = \Span\{[x,y]_t+[y,x]_t\mid x,y\in \mathrm T(J) \},
\]
where $\overline{\mathrm T(J)}=\mathrm T(J)/I_0$.

Since $\tau $ is a homogeneous homomorphism,
$\Ker \tau =
(\Ker \tau\cap J^+)\oplus (\Ker\tau\cap \mathrm S_0(J) \oplus (\Ker \tau\cap J^-)$.
Note that if $e\in J$ is a bar-unit then
$L(J)\subseteq \mathrm S_0(J)$ since $L_a = \frac{1}{2}([e^-,a^+]_t - [e^+, a^-]_t)$,
$a\in J$.

If $a^+ \in \Ker\tau\cap J^+$, $a\in J$, then
$a=\sum_i [x_i,y_i]\in [J,J]$. Since
$[L_{x_i}, y_i^+]_t = -(x_iy_i)^+$,
$[y_i^+, L_{x_i}]_t = (y_ix_i)^+$
by \eqref{eq:TKKoperations},
$a^+\in I_0$.
In the same way one may show that
$\Ker \tau\cap J^- \subseteq I_0 $.

If $U\in \mathrm \Ker\tau \cap S_0(J)$, $U=L_a + [L_b,L_c]$, then
$\tau(L_a)=\tau([L_b,L_c])= 0$ by the construction of $\mathrm S(J)$.
Since $\tau (L_a) = L_{\bar a}\in \mathrm S_0(\bar J)$, we have
$aJ,Ja \subseteq [J,J]$, in particular,
$a=ea \in [J,J] $. It is easy to see that $L([J,J])\subseteq I_0$.

If $\tau([L_b, L_c])=0 $ then $\tau([L_b, L_{ce}])=0$ since
$[L_b,L_c]-[L_b, L_{ce}]\in [L(J), L([J,J])] \subseteq I_0$.
But since $e$ belongs to the associative center of $J$,
we have $[L_b,L_{ce}]=[L_{be}, L_{ce}] = -[L_{ce}, L_{be}]$.
Therefore, $[L_b,L_c]\in \frac{1}{2} ([L_{be}, L_{ce}]
 + [L_{ce}, L_{be}]) + [L(J),L([J,J])] \in I_0$.

We have proved that $\Ker\tau = I_0$ which implies the claim.
\end{proof}

\begin{rem}
In general, this is not true that $\overline {\mathrm T(J)}
\simeq \mathrm T(\bar J)$. As an example, one may
consider $J=A^{(+)}$, where $A$ is the associative dialgebra
of the form
\[
 A =\begin{pmatrix} 0 & \Bbbk[x] & \Bbbk [x] \\
0  & 0 & \Bbbk \\
0 & 0 & 0\end{pmatrix}
\]
with the following (associative) products:
\[
a(x)\rbar b(x) = a(0)b(x), \quad a(x)\lbar b(x) = a(x)b(0),
\quad a,b \in A.
\]
Then
$[J,J] = x\Bbbk [x] e_{13}\ne 0$, where $e_{13} $ is the matrix
unit. But $J^3 = 0$, so
$I_0\cap J^\pm =0$, but $[J,J]^\pm \subseteq \Ker \tau$.
Therefore,
$\Ker \tau \not\subseteq I_0$.
\end{rem}

Since the structure Leibniz algebra $\mathrm S_0(J)$ does not depend
on the embedding of $J$ into a Jordan conformal algebra, consider
the simplest case $J\hookrightarrow (\Curr \hat J)^{(0)}$
from Theorem \ref{thm:CurrEmbedd}.
Recall that $\hat J = \bar J\oplus J$ is the split null extension
of the Jordan algebra $\bar J$, and the embedding is given by
$a\mapsto \hat a = 1\otimes \bar a + T\otimes a$, $a\in J$.

\begin{thm}
The super-structure Leibniz algebra $\mathrm T(J)$
is isomorphic to a subalgebra of $(\Curr \mathrm T(\hat J))^{(0)}$.
\end{thm}

\begin{proof}
Consider the map
$\psi: \mathrm T(J) \to (\Curr \mathrm T(\hat J))^{(0)}$ defined by
\begin{gather}\label{eq:TKK_Current}
a^\pm \mapsto \hat a^\pm = 1\otimes \bar a^\pm + T\otimes a^\pm, \quad a\in J, \\
L_a \mapsto 1\otimes \ell _{\bar a} + T\otimes \ell_a,\quad a\in J, \\
\sum_i [L_{a_i}, L_{b_i}] \mapsto \sum_i (1\otimes [\ell_{\bar a_i}, \ell_{\bar b_i}]
   + T\otimes [\ell_{\bar a_i}, \ell_{b_i}]), \quad a_i,b_i\in J.
\end{gather}
Here $\ell_x$ stands for the operator of left multiplication by $x$ in $\hat J$.

This is easy to see that \eqref{eq:TKK_Current} is a well-defined
injective map. Moreover, this is straightforward to
check that this is a homomorphism of Leibniz algebras.
\end{proof}

The following definition generalizes the similar notion for
Jordan algebras \cite{Jac}.
For a Jordan dialgebra $J$, consider the sequence
\[
J^{(1)}=J, \quad J^{(2)}=J^2, \quad J^{(n+1)}= J^{(n)}J^{(n)}
+ J(J^{(n)}J^{(n)}) + (J^{(n)}J^{(n)})J, \ n> 1.
\]
It follows from \eqref{eq:dva}, \eqref{eq:tri} that all $J^{(n)}$ are ideals of
$J$. If there exists $N\ge 1$ such that $J^{(N)}=0$ then $J$ is said to
be {\em strongly solvable} (or {\em Penico solvable}).
Note that a Jordan dialgebra $J$ is strongly solvable if and only if
so is the Jordan algebra $\bar J$.
Also, $J$ is strongly solvable (or nilpotent) if and only if $\hat J$
is strongly solvable (or nilpotent).

\begin{thm}\label{thm:Solvable}
A Jordan dialgebra $J$ is strongly solvable if and only if
$\mathrm T(J)$ is solvable.
\end{thm}

\begin{proof}
If $J$ is strongly solvable then $\hat J$ is strongly solvable,
and thus $\mathrm T(\hat J)$, the ordinary TKK
construction, is a solvable Lie algebra.
Therefore, $\Curr \mathrm T(\hat J)$ is a solvable Lie
conformal algebra. Hence, $(\Curr \mathrm T(\hat J))^{(0)}$
is a solvable Leibniz algebra that contains $\mathrm T(J)$ as a
subalgebra.

Conversely, if $\mathrm T(J)$ is solvable then so is the Lie algebra
$\overline{\mathrm T(J)}$. It was shown in the proof of
Theorem \ref{thm_barfunctor} that $\mathrm T(\bar J)$ is always a
homomorphic image of $\overline{\mathrm T(J)}$, hence, it is solvable.
Since $\mathrm T(\bar J)$ is the ordinary TKK construction for a Jordan
algebra, $\bar J$ is strongly solvable \cite{Jac}, therefore,
$J$ is strongly solvable.
\end{proof}

\begin{thm}
A Jordan dialgebra $J$ is nilpotent if and only if
$\mathrm T(J)$ is nilpotent.
\end{thm}

\begin{proof}
Let $\mathcal L$ stands for the TKK construction of an Jordan dialgebra $J$.

If $J$ is nilpotent then $\mathrm T(J) $ is nilpotent by the very
same reasons as stated in the proof of Theorem \ref{thm:Solvable}.

Conversely,
if $a\in J^n$, $n\ge 2$, then
\[
L_a \in \mathcal L_{\lceil n/2 \rceil}^0, \quad a^{\pm} \in \mathcal L_{\lceil n/2 \rceil }^{\pm},
\]
where $\lceil n/2 \rceil $ stands for the upper integral part of $n/2$.
Indeed, for $n=2$ the statement is trivial.
If it is true for all $m<n$, then for
$a\in J^{n}$, $a=bc$, $b\in J^s$, $c\in J^{n-s}$,
$s=1,\dots, n-1$, we have
\[
L_{a} = \frac{1}{2} ([b^-,c^+]_t - [b^+,c^-]_t)
\in [\mathcal L^{\lceil s/2 \rceil },\mathcal L^{\lceil (n-s)/2 \rceil } ]_t
+
[\mathcal L^{\lceil (n-s)/2 \rceil },\mathcal L^{\lceil s/2 \rceil } ]_t
\subseteq \mathcal L^{\lceil n/2\rceil }
\]
since $\lceil s/2 \rceil + \lceil (n-s)/2 \rceil \ge \lceil n/2 \rceil  $.
Similarly,
$a^\pm =
\mp [L_b, c^{\pm}]_t \in [\mathcal L^{\lceil s/2 \rceil }, \mathcal L^{
 \lceil (n-s)/2 \rceil } ]_t
\subseteq \mathcal L^{\lceil n/2 \rceil }$.

Hence, if $\mathcal L=\mathrm T(J)$ is nilpotent then so is $J$.
\end{proof}

\subsection*{Acknowledgements}
This work was partially supported by
RFBR 09-01-00157, SSc 344.2008.1,
SB RAS Integration project N~97, MD-2438.2009.1,
and
by ADTP ``Development of the Scientific Potential of Higher
School'' of the Russian Federal Agency for Education (Grant
2.1.1.419).
The authors are very grateful to Viktor Zhelyabin and Alexander Pozhidaev
for helpful discussions and valuable comments, and to the referee for pointing
out some faults in the initial version of the manuscript.

\end{document}